\newcommand{\folder}{\string~/Getextes/HeaderLiteratur/}

\documentclass[12pt,twoside,a4paper]{amsart}
\usepackage[utf8]{inputenc}

\usepackage[inner=2.5cm,outer=3cm,top=2.5cm,bottom=4cm,includeheadfoot]{geometry}

\usepackage{amsmath}
\usepackage{amssymb}
\usepackage{amsopn}
\usepackage{amsthm}
\usepackage{amsfonts}
\usepackage[right]{eurosym}
\usepackage{mathtools}

\usepackage{color}
\usepackage{graphicx}

\usepackage{tikz}
\usepackage[all]{xy}

\usepackage{longtable}
\usepackage{totpages}

\usepackage[utf8]{inputenc} 
\usepackage[T1]{fontenc}

\usepackage{mathrsfs}
\usepackage{paralist}

\usepackage{tikz}
\usetikzlibrary{matrix,arrows,calc,snakes,patterns,decorations.markings}
\usepackage{tikz-cd}

\usepackage{eucal}
\usepackage{wasysym}

\usepackage{xspace} 

\newcommand{\bC}{{\mathbb C}}

\newcommand{\bN}{{\mathbb N}}

\newcommand{\bP}{{\mathbb P}}
\newcommand{\bQ}{{\mathbb Q}}
\newcommand{\bR}{{\mathbb R}}

\newcommand{\bZ}{{\mathbb Z}}

\newcommand{\cA}{{\mathscr A}}
\newcommand{\cB}{{\mathscr B}}
\newcommand{\cC}{{\mathscr C}}

\newcommand{\cH}{{\mathscr H}}

\newcommand{\cK}{{\mathscr K}}
\newcommand{\cL}{{\mathscr L}}

\newcommand{\cX}{{\mathscr X}}

\newcommand{\dR}{{\mathcal R}}

\newcommand{\dU}{{\mathcal U}}

\renewcommand{\phi}{\varphi}

\DeclareMathOperator{\CH}{CH}
\DeclareMathOperator{\DCH}{DCH}

\DeclareMathOperator{\iso}{\cong}

\DeclareMathOperator{\inj}{\hookrightarrow}
\DeclareMathOperator{\surj}{\twoheadrightarrow}

\DeclareMathOperator{\Sym}{Sym}
\DeclareMathOperator{\too}{\longrightarrow}

\DeclareMathOperator{\Pic}{Pic} 
\DeclareMathOperator{\Div}{Div} 

\DeclareMathOperator{\NS}{NS}
\DeclareMathOperator{\del}{\partial}
\DeclareMathOperator{\birKbar}{\overline{\cB\cK}}
\DeclareMathOperator{\Mon}{Mon}
\DeclareMathOperator{\MonHdg}{Mon^2_\text{Hdg}}

\DeclareMathOperator{\Hilb}{Hilb}

\DeclareMathOperator{\UR}{\dU}
\DeclareMathOperator{\Refl}{\dR_\dU}
\newcommand{\Xtild}{\widetilde{X}}
\newcommand{\Btild}{\widetilde{B}}
\newcommand{\Xhat}{\hat{X}}

\newcommand\restr[2]{{
  \left.\kern-\nulldelimiterspace 
  #1 
  \vphantom{\big|} 
  \right|_{#2} 
  }}

\newif\ifmyversion
 \myversionfalse

\newcommand{\TODO}[1]{}

\newcommand{\Martin}[1]{}

\newdir^{ (}{{}*!/-5pt/@^{(}}  
\newdir_{ (}{{}*!/-5pt/@_{(}}

\theoremstyle{plain}
\newtheorem{proposition}{Proposition}[section]

\newtheorem{lemma}[proposition]{Lemma}

\newtheorem{corollary}[proposition]{Corollary}

\newtheorem{theorem}[proposition]{Theorem}

\theoremstyle{definition}
\newtheorem{definition}[proposition]{Definition}

\newtheorem{remark}[proposition]{Remark}

\theoremstyle{remark}

\newtheoremstyle{name}
   {}{}{\itshape}{}{\bfseries }{}{ }{\thmname{#3}.}
\theoremstyle{name}
\newtheorem{name}{bla}

\numberwithin{equation}{section}					

\begin{document}
\title[On the Beauville conjecture]{On the Beauville conjecture}
\author[U. Rie\ss]{Ulrike Rie\ss}
\address{Mathematical Institute, Endenicher Allee 60, 53115 Bonn, Germany}
\email{uriess@math.uni-bonn.de}

\begin{abstract} 
  We investigate Beauville's conjecture on the Chow ring of irreducible symplectic varieties. For
  special irreducible symplectic varieties we
  relate it to a conjecture on the existence of rational Lagrangian fibrations, which proves Beauville's
  conjecture in many new cases. We further apply the 
  same techniques to reduce Beauville's conjecture to
  Picard rank two.
\end{abstract}

\maketitle
\let\thefootnote\relax\footnotetext{Funded by the SFB/TR 45 `Periods,
moduli spaces and arithmetic of algebraic varieties' of the DFG
(German Research Foundation)}

\section{Introduction}
The object of this article is to study the following conjecture, due to Beauville in \cite{Beauville},
which deals with the Chow ring of an {\it irreducible symplectic variety} $X$, i.e.\,of a simply connected smooth
projective variety $X$, such that $H^0(X,\Omega_X^2)$ is spanned by a nowhere degenerate two-form.

\begin{name}[Conjecture (WSP)]
  For an irreducible symplectic variety $X$, let $\DCH(X) \subseteq \CH_\bQ(X)$ be the subalgebra
  generated by divisor classes. Then the restriction of the cycle class map
  $\restr{c_X}{\DCH(X)}\colon\DCH(X)\inj H^*(X,\bQ)$ is
  injective.
\end{name}

Beauville first stated this conjecture, when he was investigating the conjectural
Bloch--Beilinson filtration (see e.g. \cite[Chapter 11.2]{VoisinII}) on the Chow ring of irreducible
symplectic varieties. Motivated by the 
results about the Chow rings of K$3$ surfaces (see \cite{Beauville-Voisin}) and abelian varieties (see
\cite{Beauville1986}), Beauville asked in \cite{Beauville} whether the
conjectural Bloch--Beilinson filtration would split for irreducible symplectic varieties. He then
observed that such a splitting 
would immediately imply the injectivity of $\restr{c_X}{\DCH(X)}\colon\DCH(X)\inj H^*(X,\bQ)$, which
 is frequently called {\it weak splitting property} (thus the notation WSP). 

Beauville showed in \cite{Beauville} that WSP holds for $\Hilb^n(S)$, for a K$3$ surface $S$,
when $n=2,3$. Voisin extended this to $n\leq 2 b_2(S)_{tr}+4$, where $b_2(S)_{tr}$ is the rank of the
transcendental lattice of $S$, and further proved a generalization
of WSP for Fano varieties of lines on cubic fourfolds (see \cite{Voisin-08}).
Later, Lie Fu proved Voisin's generalization of WSP for arbitrary generalized Kummer
varieties (see \cite{Fu13}).
In an earlier paper, the author proved invariance of WSP under birational transformations (see
\cite{Riess13}). 

In this article we establish a connection of WSP to a conjecture on the existence of rational
Lagrangian fibrations (the Hasset--Tschinkel--Huybrechts--Sawon conjecture):

For an irreducible symplectic variety $X$,  denote by $\birKbar_X \subseteq
 H^{1,1}(X,\bR)$ the closure of its birational Kähler cone. 
Further let $q$ be the Beauville--Bogomolov quadratic form on the second cohomology of $X$.

\begin{name}[Conjecture (RLF)]
  Let $X$ be an irreducible symplectic variety, and suppose  $0\neq L \in \Pic(X) \cap
  \birKbar\hspace{-0.4em}_X$ satisfies $q(L)=0$. Then $L$ induces a rational Lagrangian fibration.
\end{name}

Recall that $L$ is said to induce a rational Lagrangian fibration if there exists a birational map
$f\colon X \dashrightarrow X'$ to an irreducible symplectic variety $X'$, and
  a fibration $\phi\colon X'\surj B$, to an $n$-dimensional projective base $B$,
  such that $L$ corresponds to an ample line bundle on $B$. 

In this article we prove (see Theorem \ref{thm:ConjBtoConjA}):
\begin{name}[Theorem]
  If $X$ is an irreducible symplectic variety with a line bundle $0\neq L\in
  \Pic(X)$ satisfying $q(L)=0$, then RLF for $X$  implies WSP for $X$.
\end{name}

Recently, Matsushita studied the behaviour of rational Lagrangian fibrations under deformations (see
\cite{Matsushita13}). As a
corollary, he observed that RLF holds for $X$, if $X$ is either of K$3^{[n]}$-type or deformation
equivalent to a generalized Kummer variety (see \cite[Corollary 1.1]{Matsushita13}).  

Based on this result the above Theorem implies that WSP holds in many new cases
 (see Corollary \ref{cor:ConjAforSC}): 
\begin{name}[Corollary]
  Let $X$ be an irreducible symplectic variety with $0\neq L\in \Pic(X)$ such that $q(L)=0$. If
  furthermore $X$ is of $K3^{[n]}$-type or deformation equivalent to a
  generalized Kummer variety, then
  WSP (Beauville) is satisfied for $X$.
\end{name}

In particular, this provides examples of $\Hilb^n(S)$ for K$3$ surfaces $S$, with arbitrary $n$, for
which WSP holds.

The proof of the Theorem mainly consists of two steps. In the first part (see Section
\ref{sec:redtoLB}) we show that in the given situation
 it suffices to check that $L^{\dim(X)/2+1}=0\in \CH(X)$ for all isotropic line bundles $L$. This
relies on the description of the cohomology of irreducible symplectic varieties by Verbitsky and Bogomolov, and a few
explicit computations. 
The second part (see Section \ref{sec:redtobirKbar}) reduces to the case that the line bundles lie in
fact in $\birKbar_X$. It uses results of Huybrechts, and Markman's results on the decomposition of the positive cone,
together with the 
author's earlier work on Chow rings of birational irreducible symplectic varieties.
 For these special line bundles $L^{\dim(X)/2+1}=0\in \CH(X)$, follows immediately from
RLF.

As an additional application of the techniques of this article, we show that it is enough to prove WSP for all
irreducible symplectic varieties with Picard rank two. This is presented in Section \ref{sec:redtorho2}.

\bigskip
\noindent {\bf Acknowledgements.}
I thank my advisor Daniel Huybrechts for his support. I would also like to thank Christian Lehn for his
suggestions to the first version of this article. Last but not least, I am grateful to my husband, who
never stops believing in me.

\section{Reduction to a statement on line bundles} \label{sec:redtoLB}
Throughout this article the term variety denotes a separated integral scheme of finite type over $\bC$. For an
overview on irreducible symplectic varieties, we refer to \cite{Gross-Huybrechts-Joyce}.

For an irreducible symplectic variety $X$ of dimension $2n$, denote its Beauville--Bogomolov form
on $ H^2(X,\bQ)$ by $q$, and its associated bilinear form by $(\,,\,)_q$. Let $\NS(X)_\bQ\subset
H^{1,1}(X,\bQ)$ be the N\'eron--Severi group. Furthermore, let $\cC_X \subseteq H^{1,1}(X,\bR)$ be the positive
cone with respect to $q$ (i.e.\, the connected component of $\{\alpha \in H^{1,1}(X,\bR) \mid q(a)>0\}$
containing an ample class).

By $\CH(X)$, we denote the Chow ring with integral coefficients of $X$. For $K = \bQ, \bR, \bC$, we
let $\CH_K(X)\coloneqq \CH(X)\otimes K$.

\bigskip
In \cite{Beauville}, Beauville conjectured that for any irreducible symplectic variety $X$ of dimension
$2n$ the following holds:

\begin{name}[Conjecture (WSP$(X)$)]
  Let $\DCH(X) \subseteq \CH_\bQ(X)$ be the subalgebra
  generated by divisor classes. Then the restriction of the cycle class map
  \begin{equation*}
  \restr{c_X}{\DCH(X)}\colon\DCH(X)\inj H^*(X,\bQ)
\end{equation*}
is injective.
\end{name}

Consider the following alternative conjectures on $X$:

\begin{name}[Conjecture (B$_\bC(X)$)] 
  Every $\alpha\in \CH^1_\bC(X)$ with $q(\alpha)=0$ satisfies 
  \begin{equation*} 
     \alpha^{n+1}=0 \hspace{2 em} {\text in}\ \CH^{n+1}_\bC(X).
   \end{equation*}
 \end{name}

\begin{name}[Conjecture (B$_\bR(X)$)] 
  Every $\alpha\in \CH^1_\bR(X)$ with $q(\alpha)=0$ satisfies 
$
     \alpha^{n+1}=0 \in  \CH^{n+1}_\bR(X).
 $
 \end{name}

\begin{name}[Conjecture (B$_\bQ(X)$)] 
  Every $\alpha\in \CH^1_\bQ(X) = \NS(X)_\bQ$ with $q(\alpha)=0$ satisfies 
$
     \alpha^{n+1}=0 \in  \CH^{n+1}_\bQ(X).
 $
 \end{name}

The aim of this section is to prove the following proposition:
\begin{proposition}\label{prop:BeauvilletoLB}
  Let $X$ be an irreducible symplectic variety of dimension $2n$. Then WSP$(X)$,
  B$_\bC(X)$, and B$_\bR(X)$ are equivalent. If $X$ satisfies $\del\cC_X
\cap \NS(X)_\bQ\neq 0$, then the
  above also are equivalent to B$_\bQ(X)$.
\end{proposition}

\begin{proof}[Proof of Proposition \ref{prop:BeauvilletoLB}]
  We proceed by successively proving the equivalences:

\bigskip
 \noindent
{\it WSP$(X)$ $\Leftrightarrow$ B$_\bC(X)$: }
The equivalence of WSP$(X)$ and B$_\bC(X)$ was already observed by Beauville
(see \cite[Corollary 2.3]{Beauville}). In fact, using the same techniques as in
\cite{Bogomolov96}, Beauville showed that the kernel of the map $\Sym^*(\CH^1_\bC(X))\to H^*(X,\bC)$ is
generated by $\{\alpha^{n+1} \mid q(\alpha)=0\}\subseteq \Sym^{n+1}(\CH^1_\bC(X))$. Then the equivalence
follows immediately from considering $\Sym^*(\CH^1_\bC(X))\surj \DCH_\bC(X)\to H^*(X,\bC)$.

\bigskip
 \noindent
{\it   B$_\bC(X)$ $\Leftrightarrow$ B$_\bR(X)$: }
Clearly, B$_\bC(X)$ implies  B$_\bR(X)$.

For the other implication suppose that B$_\bR(X)$ holds.
  Let $\gamma \in \CH^1_\bC(X)$ satisfy $q(\gamma)=0$. Then there exist $\alpha,\beta\in \CH^1_\bR(X)$ with $\gamma
  = \alpha + i \beta$. We need to show that $\gamma^{n+1}=0\in \CH^n_\bC(X)$. The condition $q(\alpha+i\beta)=q(\gamma)=0$ implies that $q(\alpha)=q(\beta)$ and
  $(\alpha,\beta)_q=0$. 
  Since the signature of $q$ restricted to $\NS(X)_\bQ$ is $(1,\rho(X)-1)$, one concludes that
  $q(\alpha)=q(\beta)\leq 0$. Moreover, there
  exists an element $\chi \in  \CH^1_\bR(X)$ which satisfies $-q(\chi)=q(\alpha)=q(\beta)$ and
  $(\chi,\alpha)_q=0=(\chi,\beta)_q$ (in the case $q(\alpha)=q(\beta)=0$ just choose $\chi\coloneqq \alpha$). 

  Consider the polynomial expression $P(R,S,T)\coloneqq (R\chi+S\alpha+T\beta)^{n+1} \in \linebreak \CH^{n+1}_\bR(X)[R,S,T]$
  in three independent variables $R,S,T $. Let $P'$ be the sum of all terms including an odd power of $R$
  and let $P''$ be 
  the part of even powers of $R$. Then $P=P'+P''$. Define the polynomial expression 
  $p''(S,T)\coloneqq P''(\sqrt{S^2+T^2},S,T)\in \CH^{n+1}_\bR(X)[S,T]$.

  Fix $s,t \in \bR$ and set $r\coloneqq \sqrt{s^2+t^2}$. Observe that with this choice $q(r\chi+s\alpha+t\beta)=0$.
  By assumption $0=(r\chi+s\alpha+t\beta)^{n+1}=P(r,s,t) \in \CH^{n+1}_\bR(X)$. 
  The same argument yields
  $0=(-r\chi+s\alpha+t\beta)^{n+1}=P(-r,s,t) \in \CH^{n+1}_\bR(X)$.
  This implies:
  \begin{align*}
    0&=P(-r,s,t)=P'(-r,s,t)+P''(-r,s,t)=-P'(r,s,t)+P''(r,s,t) \\
    &=-P(r,s,t)+2P''(r,s,t)
    =2\;P''(r,s,t)=2\;p''(s,t).
  \end{align*}
  Since $p''$ is a polynomial expression that vanishes for all real values, it is the zero-polynomial.
  
  Finally, consider the complexified polynomials $P_\bC,P'_\bC,P''_\bC\in \CH^{n+1}_\bC(X)[R,S,T]$, and
  $p''_\bC \in \CH^{n+1}_\bC(X)[S,T]$, and observe that:
  \begin{align*}
    \gamma^{n+1}=(\alpha + i \beta)^{n+1}= (0\chi + 1 \alpha + i \beta)^{n+1}=P_\bC(0,1,i)
    =P''_\bC(0,1,i)=p''(1,i)=0.
  \end{align*}
  Therefore, B$_\bC(X)$ follows from  B$_\bR(X)$.

\bigskip
 \noindent
{\it   B$_\bR(X)$ $\Leftrightarrow$ B$_\bQ(X)$ if $\del\cC_X
\cap \NS(X)_\bQ\neq 0$: } Once again, B$_\bQ(X)$ is a special case of  B$_\bR(X)$, so
we only need to show the other implication. This is an immediate consequence of the following two lemmas:

\begin{lemma} \label{lem:NSdense}
  If $\del\cC_X \cap \NS(X)_\bQ\neq 0$, then $\del\cC_X \cap \NS(X)_\bQ$ is dense in $\del\cC\cap\NS(X)_\bR$.
\end{lemma}
\begin{proof}
  This is known for K$3$ surfaces (cf.\,\cite{HuybrechtsK3}) and a similar argument works here.

  Fix $\alpha \in \del\cC_X \cap \NS(X)_\bQ$.  For any $\beta'\in \cC_X\cap \NS(X)_\bQ$ the element $\gamma \coloneqq  2(\alpha,
  \beta')_q\beta' - q(\beta')\alpha$ is again an element of $\del\cC_X\cap \NS(X)_\bQ$. Since the signature of
  $(\,,\,)_q$ restricted to $\NS(X)_\bQ$ is $(1,\rho(X)-1)$, where $\rho(X)$ is the rank of $\NS(X)_\bQ$, one knows that
  $(\alpha,\beta')_q\neq 0$, and thus $\alpha$ and $\gamma$ are not collinear.

  Fix a metric on the finite dimensional $\bR$-vector space $\NS(X)_\bR$. For arbitrary $\beta \in
  \del\cC_X\cap \NS(X)_\bR$ and a given $\epsilon > 0$, consider the $\epsilon$-ball
  $B_\epsilon(\beta)\subseteq \NS(X)_\bR$ around $\beta$. The cone $\{(1-r)\alpha + rb \mid r\in
  \bR_{>0}, b\in B_\epsilon (\beta)\cap 
  \del\cC_X\}$ is an open cone, and thus contains a $\beta'\in \NS(X)_\bQ$. The associated $\gamma$ lies in
  $B_\epsilon(\beta)\cap \del\cC_X \cap \NS(X)_\bQ$. This proves the lemma. 
\end{proof}

  With this lemma, the desired implication follows from taking limits:

  \begin{lemma}\label{lem:densesubset}
    Let $M\subseteq \del\cC_X\cap \NS(X)_\bR$ be a dense subset. If for every $\alpha\in M$  the
    identity $\alpha^{n+1}=0 \in \CH^{n+1}_\bR (X)$ holds, then B$_\bR(X)$ is true.
  \end{lemma}

  \begin{proof}
  Fix $\alpha \in \CH^1_\bR(X)\cap\del\cC_X$ and choose a basis $L_1,\dots, L_\rho$ of (the $\bR$-vector
  space) $\CH^1_\bR(X) =
  \NS(X)_\bR$. Since $M \subseteq \del\cC_X\cap \NS(X)_\bR$ is dense, $\alpha$ is the 
  limit of $\alpha_i \in M$. The $\alpha_i$ can be written as $\alpha_i=\sum_{j=1}^\rho {s_{i,j} L_j}$, with
  $s_{i,j}\in \bR$ converging to the coefficients of $\alpha$ for $i\to \infty$.
  In particular for a polynomial $p \in \bQ[X_1,\dots,X_\rho]$ the expression $p(s_{i,1},\dots, s_{i,\rho})$
  converges in $\bR$ for $i\to \infty$.
  Therefore, $\alpha_i^{n+1} =(\sum_{j=1}^\rho {s_{i,j} L_j})^{n+1}$ converges to $\alpha^{n+1}$ (in the
  finite-dimensional sub vector space of $\CH_\bR^{n+1}(X)$, generated by degree $n+1$ monomials in the $L_j$).
  Since by assumption $\alpha_i^{n+1}=0$, one concludes that
  $
  \alpha^{n+1}=\lim_{i\to \infty} \alpha_i^{n+1}=0 \in \CH_\bR^{n+1}(X).
  $
  \end{proof}

Together, this concludes the proof of Proposition \ref{prop:BeauvilletoLB}.
\end{proof}

\section{Reduction to elements in $\birKbar_X$}\label{sec:redtobirKbar}

 Let $X$ be an irreducible symplectic variety of dimension $2n$ as before.
Denote its Kähler cone by $\cK_X\subseteq
H^{1,1}(X,\bR)$. Define its {\it birational Kähler cone} as
  \begin{equation*}
     \cB\cK_X\coloneqq  \bigcup_f f^* (\cK_{X'})\subseteq H^{1,1}(X,\bR),
  \end{equation*}
 where the union is taken over all birational maps $f\colon X\dashrightarrow X'$ from $X$ to another irreducible
 symplectic variety $X'$. Denote its closure by $\birKbar\hspace{-0.4em}_X\subseteq H^{1,1}(X,\bR)$. Note
 that the pullback along $f\colon X\dashrightarrow X'$ is well-defined, since the indeterminacy locus is of codimension
 at least two (see e.g. \cite[Lemma 2.6]{Huybrechts:HK:basic-results}).

Consider the following conjecture on $X$:
\begin{name}[Conjecture (B$_{\cB\cK}(X)$)] 
  Every $\alpha\in \NS(X)_\bQ\cap \del\cC_X \cap \birKbar_X$  satisfies 
$
     \alpha^{n+1}=0 \in \  \CH^{n+1}_\bQ(X).
 $
 \end{name}

The aim of this section is to prove the following proposition:
\begin{proposition}\label{prop:ConjQeqConjBK}
  For an irreducible symplectic variety $X$ of dimension $2n$, B$_\bQ(X)$ and
  B$_{\cB\cK}(X)$ are equivalent.
\end{proposition}

For the proof we use earlier results of Huybrechts, Markman, and the author. 

\begin{definition}
  \begin{enumerate}
  \item For two irreducible symplectic varieties $X$ and $X'$, an isomorphism 
$\mu \colon H^2(X,\bZ)\to H^2(X',\bZ)$ is
    called {\it parallel transport operator}, if it is induced by parallel transport with respect to a
    smooth family of irreducible symplectic varieties (over a possibly singular base).

  \item In the special case $X=X'$ the subgroup of $O(H^2(X,\bZ))$, whose elements are parallel transport
    operators is called {\it monodromy group} $\Mon^2(X)$.

  \item  Let $\MonHdg(X)$ be the subgroup preserving the Hodge structure on $H^2(X,\bZ)$ (i.e.\,the subgroup
    corresponding to ``loops'' in the period domain).
  \item Particularly interesting examples of parallel transport operators are provided by parallel
    transport along a {\it degeneration of isomorphisms} $(X,X',\cX, \cX',T)$,
  i.e. along families $\cX$ and $\cX'$ over a one-dimensional smooth base $T$, with $0\in T$
  satisfying  $X\iso \cX_0$ and $X'\iso\cX'_0$, and such that $\restr{\cX}{T\setminus\{0\}} \iso
  \restr{\cX'}{T\setminus\{0\}}$ are isomorphic over $T$. 

\item  We call $(X,X',\cX, \cX',T)$ a {\it degeneration of isomorphisms  along algebraic spaces}, if $\cX$ and $\cX'$
  are the realizations as manifolds of algebraic spaces.
  \end{enumerate}
\end{definition}

  For our purpose degenerations of isomorphisms along algebraic spaces are of special interest because of the
  following theorem:
\begin{theorem}\label{thm:cycletodegofiso}
  Consider a degeneration of isomorphisms along algebraic spaces $(X,X',\linebreak[0] \cX, \cX',T)$, and let $\mu$
  be the associated parallel transport operator.
  Then there exists a cycle $Z \in
  \CH^{2n}(X\times X')$, which induces $[Z]_*=\mu\colon H^2(X,\bZ)\to H^2(X',\bZ)$, and such that  the
  correspondence $Z_*\colon\CH(X)\to  \CH(X')$ is an isomorphism of graded rings.
\end{theorem}
\begin{proof}
  Consider the Graph $\Gamma \subseteq\restr{\cX}{T\setminus\{0\}} \times_{(T\setminus\{0\})}
  \restr{\cX'}{T\setminus\{0\}}$ of the isomorphism  $\restr{\cX}{T\setminus\{0\}} \iso$ \linebreak $
  \restr{\cX'}{T\setminus\{0\}}$ and let $Z\coloneqq \restr{ \overline{\Gamma}}{\cX_0}\subseteq X\times X'$ be the special
  fibre of its closure $\overline{\Gamma}\subseteq\cX \times_{T} \cX'$. Then $[Z]_* \colon H^2(X,\bZ)\to
  H^2(X',\bZ)$ coincides with $\mu$, since $\overline{\Gamma}$ is flat over $T$.
  
  The fact, that $Z_*\colon \CH(X)\to  \CH(X')$ is an isomorphism of graded rings is an earlier result of the
  author. It is the main ingredient of the proof of \cite[Theorem 3.2]{Riess13}.
\end{proof}

  A {\it hyperkähler manifold} is a simply connected compact Kähler manifold $X$, such that $H^0(X, \Omega_X^2)$
  is generated by a nowhere degenerate holomorphic two-form.
Note that by additionally requiring projectivity, one would regain the definition of an irreducible symplectic
variety.

Recall that a {\it marking} of a hyperkähler manifold $X$ is an isometry $g\colon H^2(X,\bZ)\overset{\iso}{\too} \Lambda$
with a fixed lattice $\Lambda$.

\begin{proposition}\label{prop:nonseptodegofiso}
  If two marked irreducible symplectic varieties $(X,g)$ and $(X',g')$ correspond to non-separated points
  in the moduli space of marked hyperkähler manifolds, then they are connected by a degeneration of
  isomorphisms along algebraic spaces $(X,X',\cX,\cX',T)$, such that the parallel transport operator
  coincides with $g'^{-1}\circ g$. 
\end{proposition}

\begin{proof}
Since $(X,g)$ and $(X',g')$
correspond to non-separated points in the moduli space of marked hyperkähler manifolds, we can use the
same proof as for \cite[Proposition 2.1]{Riess13}. 
\end{proof}

 Let $f\colon X' \dashrightarrow X$ be a birational map between projective irreducible symplectic varieties. Since $f$ is defined away from a codimension two set, this induces a
  natural map $f^*\colon H^2(X,\bZ)\to H^2(X',\bZ)$ (see e.g. \cite[Lemma 2.6]{Huybrechts:HK:basic-results}).

\begin{corollary}\label{cor:Gammaf} 
  There exists a cycle $Z_f \in \CH^{2n}(X\times X')$, such that the induced map ${Z_f}_*\colon\CH(X)\to
  \CH(X')$ is an isomorphism of graded rings and the induced map $[Z_f]_*\colon H^2(X,\bZ)\to H^2(X',\bZ)$ coincides with
  $f^*$. This $Z_f$ is obtained from a degeneration of isomorphisms along algebraic spaces as in Theorem
  \ref{thm:cycletodegofiso}. 
\end{corollary}
\begin{proof}
  Fix a marking $g\colon H^2(X,\bZ) \overset{\iso}{\too} \Lambda$. Note that $f^*$ induces a marking $g'$ of $X'$ by setting
  $g'\coloneqq g\circ(f^*)^{-1}$. 
  Then \cite[Theorem 4.6']{Huybrechts:HK:basic-results} states that $(X,g)$ and $(X',g')$ correspond to
  non-separated points in the moduli space of marked hyperkähler manifolds.
  By Proposition \ref{prop:nonseptodegofiso}, $(X,g)$ and $(X',g')$ are connected by degeneration of
  isomorphism along algebraic spaces, such that the parallel transport $\mu\colon H^2(X,\bZ)\to H^2(X',\bZ)$
  coincides with $g'^{-1}\circ g$. Then Theorem \ref{thm:cycletodegofiso} provides a cycle $Z_f$ satisfying
  $[Z_f]=\mu=g'^{-1}\circ g= f^*\circ g^{-1}\circ g =f^* \colon H^2(X,\bZ)\to H^2 (X',\bZ)$.
\end{proof}

With the same methods, we can adapt a result of Huybrechts to our purposes:
\begin{proposition}\label{prop:alphaintoK}
  Let $\alpha \in \cC_X$ be a general element in $\cC_X$. Then there exists a birational irreducible symplectic
  variety $X'$, and a cycle 
  $Z\in \CH^{2n}(X\times X')$ such that $[Z]_*(\alpha)\in H^{1,1}(X',\bR)$ is a Kähler
  class. This cycle $Z$ is obtained from a degeneration of isomorphisms of algebraic spaces as in
  Theorem \ref{thm:cycletodegofiso}.
\end{proposition}
\begin{proof}
Fix a marking $g\colon H^2(X,\bZ) \iso \Lambda$. From \cite[Corollary 5.2]{Huybrechts:HK:basic-results}, there exists a marked
irreducible symplectic variety $(X',g')$ such that $g'^{-1}\circ g(\alpha)$ is Kähler. In the original
proof, this is constructed as
a degeneration of isomorphism, but the involved families are twistor families, which are
non-algebraic. For the proof of this proposition we need to find alternative families, which are realizations of
algebraic spaces.

 Since $(X,g)$ and $(X',g')$ are connected by a degeneration of isomorphisms, they
correspond to non-separated points in the moduli space of marked hyperkähler manifolds and we can apply
Proposition \ref{prop:nonseptodegofiso} to see that they are indeed connected by a degeneration of
isomorphisms along algebraic spaces. In particular the parallel transport operator $\mu\colon H^2(X,\bZ)\to H^2(X',\bZ)$
coincides with $g'^{-1}\circ g$. Theorem \ref{thm:cycletodegofiso} then yields a cycle $Z$ such that
$[Z]_*(\alpha)=\mu(\alpha)=g'^{-1}\circ g(\alpha)$ is a Kähler class.
\end{proof}

\begin{remark}
  Indeed, Huybrechts stated the more explicit condition, that
  Proposition \ref{prop:alphaintoK} holds for every $\alpha\in H^{1,1}(X,\bR)$, which satisfies
  $(\alpha,\beta)_q\neq 0$ for all integral 
  classes $\beta \in H^2(X,\bZ)$ (see \cite[p.\,503]{Huybrechts2003}).
  
  From the more recent work of Mongardi (see \cite{Mongardi13}), one can see that it even works for any
  $\alpha \in \cC_X$, which is not orthogonal to a wall divisor.
\end{remark}

Further, recall the following theorem of Markman:
\begin{theorem}[{\cite[Theorem 1.3]{Markman11}}] \label{thm:alreadyiso}
  Let $\mu \colon H^2(X,\bZ) \overset{\iso}{\too} H^2(X',\bZ)$ be a parallel transport operator, which is an isomorphism of
  integral Hodge structures. Then $\mu$ maps some Kähler
  class on $X$ again into the Kähler cone if and only if there exists an isomorphism $f\colon X'\to X$,
  such that $\mu=f^*\colon H^2(X,\bZ) \to H^2(X',\bZ)$.
\end{theorem}

Finally, we can prove the following:
\begin{proposition}\label{prop:MonOKonChow}
  Let $\mu\in \MonHdg(X)$. Then there exists a cycle $Z_\mu\in\CH^{2n}(X\times X)$, such that 
  \begin{enumerate}[(a)]
  \item \label{it:OKonH2} The induced map $[Z_\mu]_*\colon H^2(X,\bZ)\to H^2(X,\bZ)$ coincides with $\mu$.
  \item \label{it:OKonCH} The associated map ${Z_\mu}_*\colon\CH(X)\to \CH(X)$ is an isomorphism of graded rings. 
  \end{enumerate}
\end{proposition}

\begin{proof} 
  Choose a general element $\alpha\in \cK_X$ and consider its preimage $\mu^{-1}(\alpha) \in \cC_X$.
  Since $\alpha$ was chosen general, also $\mu^{-1}(\alpha)$ is a general element of $\cC_X$. 
  Using Proposition \ref{prop:alphaintoK} we observe
  that there exists a cycle $Z \in \CH^{2n}(X\times X')$, for some $X'$ birational to $X$,  which comes
  from a degeneration of isomorphisms along algebraic spaces as in Theorem \ref{thm:cycletodegofiso}, and
  such that the induced map $[Z]_*\colon H^2(X,\bR)\to H^2(X',\bR)$ satisfies
  $[Z]_*(\mu^{-1}(\alpha))\in \cK_{X'}$. 
  Since $\mu^{-1}$ and $[Z]_*$ are parallel transport operators, also the composition $[Z]_* \circ \mu^{-1}$ is a
  parallel transport operator. By 
  construction, this operator maps a Kähler class $\alpha \in \cK_X$ to the Kähler class
  $[Z]_*(\mu^{-1}(\alpha))\in \cK_{X'}$. Therefore, we can apply Theorem \ref{thm:alreadyiso} to see
  that there exists an isomorphism $f\colon X' \to X$ such that $f^*=[Z]_*\circ \mu^{-1}\colon H^2(X,\bZ)\to H^2(X,\bZ)$. The
  graph $\Gamma_f \in \CH^{2n}(X'\times X)$ of $f$ induces an isomorphism of
  Chow rings and 
  such that $f^*=[\Gamma_f]^*\colon H^2(X,\bZ)\to H^2(X',\bZ)$.
  Finally, we can observe that 
  \begin{equation*}
    \mu={f^*}^{-1}\circ [Z]_*= [\Gamma_f]_*\circ [Z]_* = [\Gamma_f\circ Z]_*\colon H^2(X,\bZ)\to H^2(X,\bZ).
  \end{equation*}
  Setting $Z_\mu\coloneqq \Gamma_f\circ Z$ immediately implies that condition \eqref{it:OKonH2} holds.

  In order to verify condition \eqref{it:OKonCH}, check it separately for $Z_*$ and ${\Gamma_f}_*$. For
  $Z_*$ this follows from Proposition \ref{prop:alphaintoK} together with Theorem
  \ref{thm:cycletodegofiso}. Conclude the proof by noticing that ${\Gamma_f}_*=f_*$ is push-forward along
  an isomorphism. 
\end{proof}

In the following we introduce notations and recall some results in order to state and prove Proposition
\ref{prop:RefltoBK}, which we will need for the proof of Proposition \ref{prop:ConjQeqConjBK}. 

\begin{definition}
  \begin{enumerate}
  \item A divisor $D \in \Div(X)$ is called {\it prime exceptional divisor} if $D$ is reduced and
    irreducible and satisfies $q(D)<0$.
  \item Denote by $\UR_X \subseteq H^{1,1}(X,\bZ)$ the set of classes of prime exceptional divisors. 
  \item For $d \in \UR_X$ define the associated reflection $R_d\in O(H^2(X, \bQ))$ as
    $R_d(\alpha)\coloneqq  \alpha - \frac{2(d,\alpha)_q}{q(d)}  d$.
  \item Let $\Refl \subseteq O(H^2(X,\bQ))$ be the subgroup generated by $\{R_d\mid d\in \UR\}$. 
  \end{enumerate}
\end{definition}

\begin{remark}
By a result of Boucksom (see \cite[Proposition 4.7]{Boucksom04}), the prime exceptional divisors are exactly the uniruled
divisors with negative Beauville--Bogomolov square.
\end{remark}
Using this observation we can reformulate a result of Huybrechts:
\begin{proposition}[{\cite[Proposition 4.2]{Huybrechts2003}}]\label{prop:birKbarFeC} 
  The following cones coincide:
  $$\birKbar_X=\{\alpha\in \overline{\cC_X}\subseteq H^2(X,\bR)\mid (\alpha,d)\geq 0 \ \forall d\in \UR_X \}.$$
\end{proposition}

For the proof of Proposition \ref{prop:RefltoBK}, we will still need the following result of Markman:
\begin{proposition}[{\cite[Proposition 6.2]{Markman11}}]
  For any $d\in \UR_X$, the reflection $R_d$ restricts to an integral morphism. As such, it is an
  element of $\MonHdg(X)$. In particular, there is an inclusion $\Refl \subseteq \MonHdg(X)$.
\end{proposition}

\begin{proposition}\label{prop:RefltoBK}
  Let $0\neq \alpha \in \NS(X)_\bQ\cap \del\cC_X$  then there exists $R\in \Refl$ such that $R(\alpha)\in
  \birKbar_X$.
\end{proposition}

\begin{proof} This proof is similar to the analogue for K$3$ surfaces, as presented in \cite{HuybrechtsK3}.

  By passing to a multiple of $\alpha$, we may assume that $\alpha \in H^2(X,\bZ)$ is an integral
  element.   Set $\alpha_0\coloneqq \alpha$.
  Fix an ample class $h\in H^2 (X,\bZ)$.
  For any element $\alpha_i\in H^2(X,\bZ)\cap \del\cC_X$, the Beauville--Bogomolov pairing
  $(\alpha_i,h)_q$ is a positive integer. If $\alpha_i\notin \birKbar_X$, then by Proposition
  \ref{prop:birKbarFeC} there exists $d_i\in \UR_X$ with $(\alpha_i,d_i)_q <0$.
  Note that $(d_i,h)>0$ and $(\alpha_i,h)>0$, since $h$ is ample. 
  Set $\alpha_{i+1}\coloneqq R_{d_i}(\alpha_i)$, and observe that
  \begin{align*}
    (\alpha_{i+1},h)_q
    &= \big(R_{d_i}(\alpha_i),h \big)_q
    = \big(\alpha_i- \frac{2(d_i,\alpha_i)_q}{q(d_i)}  d_i\, ,h \big)_q
    =(\alpha_i,h)_q - 2\underbrace{\frac{2(d_i,\alpha_i)_q}{q(d_i)}\cdot (d,h)_q}_{>0} \\[-1em]
    &< (\alpha_i,h)_q \,.
  \end{align*}
  If $\alpha_{i+1}\notin \birKbar$, repeat the above for $\alpha_{i+1}$. 
  Since $(\alpha_0,h)> (\alpha_1,h)>(\alpha_2,h)> \dots$ is a descending sequence of positive integers,
  this procedure needs to stop for some $k\in \bN$, which implies that $\alpha_k\in \birKbar_X$. Set
  $R\coloneqq  R_{d_k}\circ R_{d_{k-1}} \circ \dots \circ R_{d_0}\in \Refl$. This concludes the proof, since
  $R(\alpha)=\alpha_k \in \birKbar_X$. 
\end{proof}

\begin{proof}[Proof of Proposition \ref{prop:ConjQeqConjBK}]
  Since B$_{\cB\cK}(X)$ is a special case of B$_\bQ(X)$, we just need to prove the
  other implication. Let  $0 \neq\alpha\in \NS(X)_\bQ$ with $q(\alpha)=0$, i.e. $\alpha \in \NS(X)_\bQ\cap \cC_X$. By
  Proposition \ref{prop:RefltoBK} there exist $R\in \Refl\subseteq \MonHdg(X)$ such that
  $R(\alpha)\eqqcolon\beta\in \birKbar_X$.  Associate to $R^{-1}\in \MonHdg(X)$  a cycle $Z\in \CH^{2n}$ as in
  Proposition \ref{prop:MonOKonChow}.

  Assuming that B$_{\cB\cK}(X)$ holds, implies that $\beta^{n+1}=0\in
  \CH_\bQ^{n+1}(X)$. Then  
  \begin{equation*}
    \alpha^{n+1}=\big(R^{-1}(\beta)\big)^{n+1}
    =\big({Z}_*(\beta)\big)^{n+1}
    ={Z}_*(\underbrace{\beta^{n+1}}_{=0})=0.
  \end{equation*}
This proves Proposition \ref{prop:ConjQeqConjBK}.
\end{proof}

\section{Main result}
This section contains the main result of this article, which relates the conjectures WSP and RLF. We
deduce that WSP holds in many known cases.

Let as before $X$ be a $2n$-dimensional irreducible symplectic variety.
\begin{definition}
  A line bundle $L\in \Pic(X)$ is said to {\it induce a rational Lagrangian fibration} if there exists a
  birational map $f\colon X \dashrightarrow X'$ to an irreducible symplectic variety $X'$, and 
  a fibration $\phi\colon X'\surj B$ to a (possibly singular) $n$-dimensional projective base $B$, such
  that $L$ corresponds to an ample line bundle on $B$.
\end{definition}

The following is known as Hasset--Tschinkel--Huybrechts--Sawon conjecture:
\begin{name}[Conjecture (RLF$(X)$)]
  Suppose  $0\neq L \in \Pic(X) \cap
  \birKbar\hspace{-0.4em}_X$ satisfies $q(L)=0$. Then $L$ induces a rational Lagrangian fibration.
\end{name}

We can now formulate the main result of this article:
\begin{theorem} \label{thm:ConjBtoConjA}
  For any irreducible symplectic variety $X$ with $\NS(X)_\bQ\cap \del\cC_X \neq 0$, RLF$(X)$ implies WSP$(X)$.
\end{theorem}

For the proof we will use the following Lemma:
\begin{lemma}\label{lem:pullbackfromsingular}
Let $X'$ be an irreducible symplectic variety, and consider a Lagrangian fibration $X'
\overset{\phi}{\too} B$ over a possibly singular base $B$. Any line bundle $M \in \Pic(B)$ satisfies 
\begin{equation*}
  \big(\phi^*(M)\big)^{n+1}=0 \in \CH^{n+1}(X').
\end{equation*}
\end{lemma}
\begin{proof}
  Fix $M\in \Pic(B)$. Consider a desingularization $\Btild\overset{r}{\too} B$, and set $\Xtild\coloneqq X' \times_B
  \Btild$. Let $\Xhat$ be a variety obtained from $X'$ by a sequence of blow-ups in smooth loci, which
  allows for a map $\Xhat\to \Xtild$. Then there is a commutative diagram 
$$
 \begin{tikzcd}
   \Xhat \ar{r}{r'} \ar{d}[swap]{\phi'} &X' \ar{d}{\phi}\\
   \Btild \ar{r}{r} & B \, ,
 \end{tikzcd}
$$
and therefore $\phi'^*r^*(M)=r'^*\phi^*(M) \in \Pic(\Xhat)$.

Since $\Btild$ is smooth of dimension $n$, its Chow groups are endowed with a multiplicative structure,
and therefore $\big(r^*(M)\big)^{n+1}=0 \in \CH^{n+1}(\Btild)$.

Consequently:
\begin{equation*}
  r'^*\Big(\big(\phi^*(M)\big)^{n+1}\Big)=\big(r'^*\phi^*(M)\big)^{n+1}
  =\big(\phi'^*r^*(M)\big)^{n+1}=\phi'^*\Big(\underbrace{\big(r^*(M)\big)^{n+1}}_{=0}\Big)
  =0.
\end{equation*}
Complete the proof by observing that $r'^*$ is injective, since $r'$ is a sequence of blow-ups of smooth
varieties in smooth loci (see \cite[Proposition 6.7.(b)]{Fulton}).
\end{proof}

\begin{proof}[Proof of Theorem \ref{thm:ConjBtoConjA}]
  The theorem follows from Proposition \ref{prop:BeauvilletoLB} and Proposition
  \ref{prop:ConjQeqConjBK}, by observing that RLF$(X)$ implies B$_{\cB\cK}(X)$:

  Suppose that RLF$(X)$ holds. Every $\alpha\in \NS(X)_\bQ\cap \del\cC_X \cap \birKbar_X$ has a multiple
  $L\in \Pic(X)$. Assuming RLF$(X)$ implies that $L$ induces a rational Lagrangian fibration
  $X\overset{f}{\dashrightarrow}X'\overset{\phi}{\too}B$. By means of Corollary \ref{cor:Gammaf}, associate to $f$ a
  cycle $Z_f$, with ${Z_f}_*=f^*\colon\Pic(X')\to \Pic(X)$.
  Then there exists an ample line bundle $M\in \Pic(B)$ such that
  \begin{equation*}
    L^{n+1}=\Big(f^*\circ \phi^*(M)\Big)^{n+1}
    =\Big({Z_f}_*\circ \phi^*(M)\Big)^{n+1}
    \overset{(*)}{=}{Z_f}_*\Big(\underbrace{\big(\phi^*(M)\big)^{n+1}}_{=0}\Big)
    =0.
  \end{equation*}
   Note that the vanishing follows from Lemma \ref{lem:pullbackfromsingular}, and that $(*)$ holds, because ${Z_f}_*$ is an isomorphism of graded rings by Theorem \ref{thm:cycletodegofiso}.
   This proves the theorem.
\end{proof}

\begin{corollary}\label{cor:thmforbigrho}
  For an irreducible symplectic variety $X$ with $\rho(X)\geq 5$, RLF$(X)$ implies WSP$(X)$.
\end{corollary}
\begin{proof}
    Observe that by Hasse--Minkowski (see e.g. \cite[IV.3: Corollary 2]{Serre73}) any indefinite lattice
    (over $\bQ$) of rank $\geq 5$ contains a square-zero element. Therefore the corollary is just a
    special case of  Theorem \ref{thm:ConjBtoConjA}.
\end{proof}

In the following, we state a recent result of Matsushita, to which we can apply the above results, in
order to prove WSP in many cases. 

  An irreducible symplectic variety is said to be {\it of K$3^{[n]}$-type} if it is deformation equivalent to a
  Hilbert scheme of $n$ points on a K$3$ surface.
  Another series of examples of irreducible symplectic varieties is provided by generalized Kummer
  varieties (see \cite[Section 7]{Beauville1983}) and their deformations. 

Recently Matsushita showed that if a line bundle $L\in \Pic(X)$ induces a rational Lagrangian
fibration over $\bP^n$, then for any deformation $(X_t,L_t)$ of the pair $(X,L)$ with $L_t\in \birKbar_{X_t}$,
also $L_t$ induces a rational Lagrangian fibration over $\bP^n$ (see
\cite[Theorem1.2]{Matsushita13}). As a
corollary, he observed:
\begin{proposition}[{\cite[Corollary 1.1]{Matsushita13}}]\label{prop:ConjBforSC}
  If $X$ is either of K3$^{[n]}$-type or deformation equivalent to a generalized Kummer, then RLF$(X)$ holds.
\end{proposition}

\begin{remark}
  In the special case of moduli spaces of sheaves on K3 surfaces, this was already shown by Bayer and
  Macr\`i in \cite[Remark 11.4]{BayerMacri13}. And for more general ``non-special''
  K3$^{[n]}$-type varieties $X$, it was first proved by Markman in \cite[Theorem
  6.3]{Markman13}.
  Both results build upon work of Markushevich (see \cite{Markushevich06}) and Sawon (see \cite{Sawon07}). 
  Matsushita's proof relies on \cite{Markman13} and \cite{Yoshioka12}.
\end{remark}

As an application of the main theorem, we can deduce:
\begin{corollary} \label{cor:ConjAforSC} 
Let $X$ be an irreducible symplectic variety with $\NS(X)_\bQ\cap \del\cC_X \neq 0$. If $X$ is of
  $K3^{[n]}$-type or deformation equivalent to a generalized Kummer variety, then WSP$(X)$
  (Beauville) is satisfied.
  In particular WSP$(X)$ holds for all $X$ of these deformations types which satisfy $\rho(X)\geq 5$.
\end{corollary}
 
\begin{proof}
  Apply Theorem \ref{thm:ConjBtoConjA} and Corollary \ref{cor:thmforbigrho} to Proposition \ref{prop:ConjBforSC}.
\end{proof}

\section{Reduction of Beauville's conjecture (WSP) to Picard rank two}\label{sec:redtorho2}
In this section we show that it is enough to prove Beauville's conjecture (WSP) for irreducible symplectic
varieties $Y$ with $\rho(Y)=2$.

Let $X$ be an arbitrary irreducible symplectic variety of dimension $2n$.
\begin{lemma}\label{lem:rho2dense}
  Fix an ample line bundle $H\in \Pic(X)$.
  The subset $M\coloneqq \{\alpha\in \del\cC_X\cap\NS(X)_\bR \mid \alpha = r H + s L,\ L\in \NS(X)_\bQ,\ r,s\in \bR \} \subseteq \del
  \cC_X\cap\NS(X)_\bR$ is dense. 
\end{lemma}
\begin{proof}
Fix an arbitrary open set $U\subseteq \del\cC_X\cap\NS(X)_\bR$. 
The open cone $\{(1-r)H + ru \mid r\in \bR_{>0}, u\in U\}$ contains a rational element $L\in \NS(X)_\bQ$.
There is a unique $r\in \bR_{>0}$ such that $\alpha_r\coloneqq(1-r)H+rL$ lies in $\del\cC_X \cap
\NS(X)_\bR$. By construction this $\alpha_r$ lies in $U \cap M$.
\end{proof}

In order to reduce to Picard rank two, we will use \cite[Theorem 10.19]{VoisinII}, which
is originally due to Bloch and Srinivas (\cite{Bloch80}, \cite{Bloch-Srinivas83}).
In fact, we need the following slightly more general result:

\begin{proposition}\label{prop:Bloch-Srinivas-general}
  Let $f\colon\cX \to T$ be a smooth projective morphism between smooth
  varieties. Fix a cycle $Z\in \CH^k(\cX)$.
  Suppose that there exists a subvariety $\cX'\subseteq \cX$, such that for general $t\in T$
  the restricted cycle $Z_t\coloneqq \restr{Z}{\cX_t}\in \CH^k(\cX_t)$ satisfies $Z_t=0\in \CH^k(\cX_t \setminus \cX_t')$.

  Then there exists $m\in \bZ_{>0}$, a cycle $Z'$ supported in $\cX'$, a proper closed algebraic subset
  $T'\subsetneq T$, and a cycle $Z''$ supported in $f^{-1}(T')$, which satisfy
  \begin{equation*}
    mZ=Z'+Z''\in \CH^k(\cX).
  \end{equation*}
\end{proposition}

\begin{proof}
  The statement of \cite[Theorem 10.19]{VoisinII} only differs in one point: There, it is required that
  $Z_t=0\in \CH^k(\cX_t\setminus \cX_t')$ for all $t\in T$, not just for a general element. However, the
  proof presented in \cite{VoisinII}, already proves the slightly more general statement above.
\end{proof}

\begin{proposition}
  Suppose that for all irreducible symplectic varieties $Y$ satisfying $\rho(Y)=2$ WSP$(Y)$
  holds. Then WSP holds for all irreducible symplectic varieties.
\end{proposition}

\begin{proof}
  Let $X$ be an arbitrary irreducible symplectic variety of dimension $2n$, and fix an ample line bundle $H$. 
  In Section \ref{sec:redtoLB} we showed the equivalence of WSP$(X)$ with B$_\bR(X)$ (see
  Proposition \ref{prop:BeauvilletoLB}). Additionally we showed (see Lemma \ref{lem:densesubset}) that in
  order to prove  B$_\bR(X)$ it is enough to find a dense subset
  $M\subseteq\del\cC_X\cap \NS(X)_\bR$ such that the identity $\alpha^{n+1}=0\in \CH^{n+1}_\bR(X)$ is
  satisfied for all $\alpha\in M$.
  By Lemma \ref{lem:rho2dense}, we can chose $M$ to be $\{\alpha\in \del\cC_X\cap\NS(X)_\bR\mid \alpha =
  r H + s L,\ L\in \NS(X)_\bQ,\ r,s\in \bR \}$.

  Fix $\alpha=rH+sL\in M \subseteq \del \cC_X\cap\NS(X)_\bR$, together with suitable $r,s\in \bR$, and
  $L\in \Pic(X)$. We only need to show, that $\alpha^{n+1}=0\in \CH^{n+1}_\bR(X)$.

  Choose an algebraic family $\cX$ over a one-dimensional smooth base $T$, which satisfies 
  \begin{enumerate}[(a)]
  \item there is a point $0\in T$ such that $\cX_0=X$,
  \item all fibres $\cX_t$ are smooth,
  \item the line bundles $H$ and $L$ deform in this family, i.e. there exist line bundles $\cH,\cL\in
    \Pic(\cX)$, such that $\cH_0\coloneqq \restr{\cH}{\cX_0}=H$ and $\cL_0\coloneqq \restr{\cL}{\cX_0}=L$, and
  \item for general $t \in T$ the Picard rank is $\rho(\cX_t)=2$.
  \end{enumerate}

  Consider the class $\cA\coloneqq r\cH+ s \cL\in \CH^1_\bR(\cX)$, and note that
  $\cA_0\coloneqq \restr{\cA}{\cX_0}=\alpha\in \CH^1_\bR(X)$. Let $t\in T$ be such that $\cX_t$ is an
  irreducible symplectic variety with $\rho(\cX_t)=2$. Then, by the assumption, WSP$(\cX_t)$ and thus
  B$_\bR(\cX_t)$ holds. Since the Beauville--Bogomolov form is invariant under
  deformations, $q(\cA_t)=q(\alpha)=0$. Therefore, B$_\bR(\cX_t)$ implies that $\cA_t^{n+1}=0\in
  \CH^{n+1}_\bR(\cX_t)$. 

  Apply Proposition \ref{prop:Bloch-Srinivas-general} with $Z\coloneqq \cA^{n+1}$ and $\cX'\coloneqq
  \emptyset$, in order to 
  conclude that there exists $m\in \bZ_{>0}$ and a cycle $Z''\in \CH^{n+1}_\bR(\cX)$ supported over
  finitely many points in $T$, such that $m\cA^{n+1}=Z''$. Pulled back to the special fibre, this gives:
  \begin{equation*}
    \alpha^{n+1}=\frac{m}{m}\cA_0^{n+1}=\frac{1}{m}\restr{Z''}{\cX_0}=0\in \CH^{n+1}_\bR(X),
  \end{equation*}
  which is all we needed to show.
\end{proof}

\bibliographystyle{alpha}
\bibliography{\folder Literatur}

\end{document}